\documentclass[10pt]{amsart}
\usepackage{amsmath}
\usepackage{amssymb}
\usepackage{latexsym}

\newtheorem{theorem}{Theorem}[section]

\newtheorem{lemma}[theorem]{Lemma}

\newtheorem{prop}[theorem]{Proposition}
\theoremstyle{remark}\newtheorem{rem}[theorem]{Remark}
\newenvironment{proof*}{\vskip 2mm\noindent {}}{\hfill $\Box$ \vskip 2mm}
\numberwithin{equation}{section}
\newcommand{\CC}{{\mathbb{C}}}
\newcommand{\DD}{{\mathbb{D}}}
\newcommand{\GG}{{\mathbb{G}}}

\newcommand{\RR}{{\mathbb{R}}}

\newcommand{\BB}{{\mathbb{B}}}
\newcommand{\TT}{{\partial \DD}}
\newcommand{\EE}{{\mathbb{E}}}
\newcommand{\eps}{\varepsilon}
\newcommand{\la}{\lambda}
\newcommand{\id}{\operatorname{id}}
\newcommand{\Aut}{\operatorname{Aut}}
\newcommand{\va}{\varphi}
\newcommand{\oo}{\omega}
\newcommand{\be}{\beta}
\newcommand{\ga}{\gamma}
\newcommand{\de}{\delta}
\def\re{\operatorname{Re}}
\def\Re{\operatorname{Re}}

\renewcommand{\phi}{\varphi}

\begin{document}

%%%%%%%%%%%%%%%%%%%%%%%%%%%%%%%%%%%%%%%%%%%%%%%%%%%%%%%%%%%%%%%%%%%%%%

%%%%%%%%%%%%%%%%%%%%%%%%%%%%%%%%%%%%%%%%%%%%%%%%%%%%%%%%%%%%%%%%%%%%%%

\title[Uniqueness of left inverses]{Nevanlinna-Pick problem and uniqueness of left inverses in convex domains, symmetrized bidisc and tetrablock}

\address{Institute of Mathematics, Faculty of Mathematics and Computer Science, Jagiellonian University,
\L ojasiewicza 6, 30-348 Krak\'ow, Poland}
\thanks{The first author is partially supported by the Polish Ministry of Science and Higher Education grant Iuventus Plus IP2012 032372.
The second author is partially supported by the grant of the Polish National Science Centre no. UMO–2011/03/B/ST1/04758.}

\author{\L ukasz Kosi\'nski}

\email{{Lukasz.Kosinski}@im.uj.edu.pl}

\author{W\l odzimierz Zwonek}

\email{{Wlodzimierz.Zwonek}@im.uj.edu.pl}

%    General info


%%%\dedicatory{This paper is dedicated to }


\begin{abstract}
In the paper we discuss the problem of uniqueness of left inverses (solutions of two point Nevanlinna-Pick problem)
in bounded convex domains, strongly linearly convex domains, the symmetrized bidisc and the tetrablock.
\end{abstract}

\maketitle

\section{Motivation: geodesics vs. left inverses}

The problem we are dealing with has two origins. The one is connected with the equality of the Lempert
function and the Carath\'eodory distance of two different points $w$, $z$ in a taut domain $D$
which is equivalent to the existence of holomorphic
functions: $f:\DD\to D$ and $F:D\to\DD$ such that $f$ passes through $w$ and $z$ and $F\circ f=\id_{\DD}$.
We call such an $f$ {\it a complex geodesic} and
$F$ {\it a left inverse (of $f$)}. The most general result in this direction is the Lempert Theorem on equality of the Lempert function
and the Carath\'eodory distance in the class of strongly linearly convex domains (and simultaneously the uniqueness
of the complex geodesics in that class of domains) or convex domains (in general without the uniqueness of geodesics).
The other origin is the Nevanlinna-Pick problem which
extensively studied in the case of the unit disc has also been recently
investigated in higher dimensional domains, especially in the polydisc.
In our paper we study the Nevanlinna-Pick problem for two points in more general higher dimensional (of dimension at least two) domains.

To illustrate the problem let us present two examples:
the Euclidean (two-dimensional) ball and the bidisc.
Complex geodesics in the Euclidean ball are uniquely determined whereas in the bidisc they are (generically) non-unique.
It is natural that we may study the problem
of uniqueness of left inverses. In the case of the Euclidean ball the left inverses are non-uniquely determined. Actually, for
the complex geodesic $\DD\owns\lambda\mapsto(\lambda,0)\in\BB_2$ the functions $\BB_2\owns z\mapsto\frac{z_1}{\sqrt{1+\gamma z_2^2}}\in\DD$,
where $\gamma\in\overline\DD$ are left inverses. On the other hand the non-uniquely determined geodesics for points $(0,0)$
and $(\lambda,\gamma\lambda)$ in the bidisc, where $|\gamma|<1$, determine uniquely
the left inverse ($\DD^2\owns z\mapsto z_1\in\DD$). But the points $(0,0)$ and $(\lambda,\gamma\lambda)$, $|\gamma|=1$
(with uniquely determined geodesics) have many left inverses:
for instance, $\DD^2\owns z\mapsto tz_1+(1-t)\bar\gamma z_2\in\DD$, $t\in[0,1]$.
Here we see a correspondence: uniqueness of geodesics corresponds to
non-uniqueness of left inverses. This may be rephrased as follows: a complex geodesic
$f:\DD\to\DD^2$ has a non-uniquely determined left inverse if and only if both $f_1$ and $f_2$ are automorphisms of the unit disc (this is actually mentioned in \cite{Agl-McC2003} while introducing a notion of a balanced pair).
This is the problem that is a starting one in our paper.

In most cases studied by us
there is such a correspondence. More precisely,
we show that in the class of strongly linearly convex domains any complex geodesic has many left inverses
(see Theorem~\ref{theorem:strongly}) whereas in the case of convex domains the non-uniqueness of geodesics
implies the uniqueness of left inverses (this is exactly the case in two
dimensional case,
in higher dimension there must be many more geodesics - see Theorem~\ref{theorem:uniqueness}).
Recall that strongly linearly convex domains admit uniquely determined geodesics. There are, however, other convex domains
with uniquely determined complex geodesics and yet determining uniquely left inverses (see Proposition~\ref{prop:example}).

The above discussion and the results presented show that the situation in the case of convex domains is relatively well understood
and to some extent regular. We may go, however, beyond that class of domains. Recall that there are two $\CC$-convex domains
not biholomorphic to convex ones for which the Lempert Theorem holds (i.e. the Carath\'eodory function and the Lempert function coincide).
These domains although somehow related turn out to be quite different as the uniqueness of the geodesics is concerned. The symmetrized bidisc
has uniquely determined geodesics and yet many of them (but not all!) admit unique left inverses
(see Theorem~\ref{theorem:symmetrizedbidisc-leftinverse}) whereas
a generic geodesic of the tetrablock is non-unique
 but the problem of uniqueness of left inverses is a much more complex one (see Section~\ref{section:tetrablock}).
We are able to determine the Carath\'eodory set for the tetrablock and we can settle down the problem but the philosophy which is behind that phenomenon is still a mystery to us.

Note that the study of the problem of uniqueness of left inverses in the symmetrized bidisc reduces,
in some cases, to the uniqueness of the Nevanlinna-Pick problem for three points in the bidisc.
Moreover, in our study in these two domains we rely very much upon the description of holomorphic retracts in the bidisc (see Theorem~\ref{theorem:retracts}).

Although the results presented in the paper allow us to understand the phenomenon of uniqueness of left inverses the form of left inverses in the non-uniqueness case is not understood at all. The authors have a very vague idea that a notion of a magic function (see \cite{Agl-You2008}) could be a useful tool in handling that problem.

\section{Definitions, notations and known results}
To start with let us recall basic definitions, notation, facts and results that we shall use.

For a domain $D\subset\CC^n$, $w,z\in D$ we define two holomorphically invariant functions.  {\it The Lempert function} is defined as follows
\begin{equation}
 \tilde k_D(w,z):=\inf\{p(\lambda_1,\lambda_2)\},
\end{equation}
where the infimum is taken over all $\lambda_1,\lambda_2\in\DD$ and holomorphic mappings $f:\DD\to D$ such that $f(\lambda_1)=w$, $f(\lambda_2)=z$. Here $p$ denotes the Poincar\'e distance on the unit disc $\DD$.

{\it The Carath\'eodory pseudodistance} is defined by the formula
\begin{equation}
 c_D(w,z):=\sup\{p(F(w),F(z))\},
\end{equation}
where the supremum is taken over all holomorphic functions $F:D\to\DD$. Recall that $c_D\leq\tilde k_D$. The equality
$\tilde k_D(w,z)=c_D(w,z)$ for some $w,z\in D$, $w\neq z$ is closely related to the existence of {\it complex geodesics}, i.e. holomorphic
functions $f:\DD\to D$ such that $w=f(\lambda_1^0)$, $z=f(\lambda_2^0)$ and $c_D(w,z)=p(\lambda_1^0,\lambda_2^0)$. Recall that
if $f$ is a complex geodesic then so is $f\circ a$ for any $a\in\Aut(\DD)$, moreover in this case
$c_D(f(\lambda_1),f(\lambda_2))=\tilde k_D(f(\lambda_1),f(\lambda_2))=p(\lambda_1,\lambda_2)$ for any $\lambda_1,\lambda_2\in\DD$.
Recall that in the case of $D$ being taut for any pair of different points $w,z\in D$
there is a function for which the infimum in the definition of the Lempert function is attained.
In our paper we shall consider only such domains (more precisely,  bounded convex, strongly linearly convex or $\CC$-convex).

To present the basic result on the equality of the Lempert function and the Carath\'eodory pseudodistance
recall that the smoothly bounded domain $D\subset\CC^n$ is called {\it strongly linearly convex}
if there is a smooth defining function $r$ of the domain $D$ such that
\begin{equation}
 \sum\sb{j,k=1}\sp{n}\frac{\partial^2 r}{\partial z_j\partial\bar z_k}(z)X_j\bar X_k>\Big|\sum\sb{j,k=1}\sp{n}
\frac{\partial^2 r}{\partial z_j\partial z_k}(z)X_jX_k\Big|
\end{equation}
for any $z\in\partial D$ and any non-zero vector $X$ from the complex tangent space to $\partial D$ at $z$.

\begin{theorem}\label{theorem:lempert} {\rm (see \cite{Lem1981}, \cite{Lem1984})}
Let $D$ be a bounded convex domain in $\CC^n$, $n\geq 1$ or
let $D$ be a strongly linearly convex domain in $\CC^n$, $n>1$. Then $\tilde k_D=c_D$.

Moreover, for any points $w,z\in D$, $w\neq z$ there is
a complex geodesic passing through $w$ and $z$.
In the case of strongly linearly convex domain the geodesics passing through a given pair of points are unique
(up to an automorphism of $\DD$).
\end{theorem}

In the case when $f:\DD\to D$ is a complex geodesic there is a holomorphic function $F:D\to\DD$ such that $F\circ f$
is the identity of the unit disc. Note also that in the class of taut domains the fact that a holomorphic mapping $f:\DD\to D$ is a complex geodesic is actually equivalent to the existence of a holomorphic function $F:D\to\DD$ such that $F\circ f=\id_{\DD}$. In any case we call such an $F$ {\it the left inverse to $f$}.

For basic properties of the Lempert function and the Carath\'eodory distance and notions related to them
we refer the Reader to \cite{Jar-Pfl1993}.

Recall that recently two non-trivial domains which could not be directly handled by the Lempert theory turned out to satisfy the equality of the Lempert function and the Carath\'eodory distance. These are {\it the symmetrized bidisc} $\GG_2$ defined as the image of $\DD^2$ under the symmetrization mapping $$\pi:\DD^2\owns(\lambda_1,\lambda_2)\mapsto (\lambda_1+\lambda_2,\lambda_1\lambda_2)$$ and {\it the tetrablock} $\EE$ defined as follows
\begin{equation}
 \EE:=\{x\in\CC^3:|x_1-\bar x_2x_3|+|x_2-\bar x_1x_3|+|x_3|^2<1\}.
\end{equation}
Recall that both domains are {\it $\CC$-convex} (i.e. their intersection with any complex line is
connected and simply connected), non-biholomorphically equivalent to convex domains and the Lempert function
and Carath\'eodory distance coincide on them (see \cite{Agl-You2004}, \cite{Cos2004}, \cite{Abo-Whi-You2007},
\cite{Edi-Kos-Zwo2012}, \cite{Nik-Pfl-Zwo2008}, \cite{Zwo2013}).

\bigskip

Let us fix a domain $D\subset\CC^n$. Consider the following {\it Nevanlinna-Pick problem}: given $N$ points
(called {\it nodes}) $z_1,\ldots,z_N\in D$, and $N$ complex numbers
$\lambda_1,\ldots,\lambda_N$ (called {\it targets}) decide whether there exists a holomorphic function
$F:D\to\DD$ such that $F(z_j)=\lambda_j$. In case the function $F$ is such that the supremum
norm $||F||$ is one and there is no solution of the Nevanlinna-Pick problem of the supremum norm smaller than one
the solution $F$ is called {\it extremal}. Note that in the case when $N=2$
and the equality $\tilde k_D(w,z)=c_D(w,z)=p(\lambda_1,\lambda_2)$
(with holomorphic $f:\DD\to D$ such that $f(\lambda_1)=w$, $f(\lambda_2)=z$) holds, the
(extremal) Nevanlinna-Pick problem for $w,z$ and $\lambda_1,\lambda_2$ reduces to finding the left inverse to $f$.
In our paper we shall deal with this special Nevanlinna-Pick problem. We shall study the problem of uniqueness
of the extremal solution in a wide class of domains for which the Lempert Theorem holds
(bounded convex, strongly linearly convex domains, the symmetrized bidisc and the tetrablock).
Among others we shall see that the two-point Nevanlinna-Pick problem in the symmetrized bidisc sometimes
reduces to the three point Nevanlinna-Pick problem in the bidisk.
There are many results on the Nevanlinna-Pick problem on the polydisc:
among others the existence of the solutions, the uniqueness of the extremal problem
and  the structure of the set of uniqueness is studied
(see e. g. \cite{Agl-McC2000}, \cite{Agl-McC2002}, \cite{Bal-Tre1998}, \cite{Guo-Hua-Wan2008},  \cite{Sch2011}).

Note that in the case $N=2$, for $D$ as above and for pairs $(w,z)$ for which the Carath\'eodory and Lempert function coincide
the Nevanlinna-Pick problem is actually the problem of finding the left inverses to a complex geodesic passing through $w$ and $z$.

The results of the paper are the following. In Section~\ref{section:non-uniqueness} we show that
the strongly linearly convex domains admit non-uniqueness of left inverses to all complex geodesics (Theorem~\ref{theorem:strongly}). In Section~\ref{section:uniqueness} we provide a sufficient condition for
uniqueness of left inverses under the assumption of existence of 'many' (in case $n=2$ two) left inverses passing through
a given pair of points (Theorem~\ref{theorem:uniqueness}) and we show that non-uniqueness of complex geodesics
is not necessary for uniqueness of left inverses
(see Proposition~\ref{prop:example}). In Sections~\ref{section:symmetrizedbidisc} and \ref{section:tetrablock}
we discuss the problem of the uniqueness
of left inverses in two examples of $\CC$-convex domains for which the Lempert theory does not apply directly and thus the study of the uniqueness of complex geodesics and left inverses must be very specific.
Both examples deliver unexpected phenomena and show its close connection with the Nevanlinna-Pick problem in the bidisc (see
Theorem~\ref{theorem:symmetrizedbidisc-leftinverse} and results of Section~\ref{section:tetrablock}).

\section{Non-uniqueness of left inverses in strongly linearly convex domains}\label{section:non-uniqueness}

We start our study with presenting a result on non-uniqueness of left inverses for complex geodesics
in the case of strongly linearly convex domains. The proof is based on a method of Lempert which enables us to
reduce the problem to the same problem in the Euclidean unit ball.
\begin{theorem}\label{theorem:strongly}
Assume that $D$ is smooth strongly linearly convex domain of $\CC^n$, $n>1$. Then left inverses are never uniquely defined.
\end{theorem}

\begin{proof}
Let $f$ be a complex geodesic in $D.$ It follows from \cite{Lem1984} (Proposition 11) that there exist a smooth domain $G\subset \CC^n$ and a biholomorphism $\Phi: D\to G$ such that
\begin{itemize}
\item $\Phi(D)=G$;
\item $g(\zeta):=\Phi(f(\zeta))=(\zeta,0,\ldots,0),\ \zeta\in\overline \DD$;
\item $\nu_G(g(\zeta))=(\zeta,0,\ldots,0),\ \zeta\in\partial\DD$, where $\nu_G$ denotes the outer unit vector to $\partial G$;
\item for any $\zeta\in\partial\DD$, the point $g(\zeta)$ is a point of the strong linear convexity of $G$.
\end{itemize} Moreover, it follows immediately from Lempert's construction that $G$ may be chosen so that
\begin{itemize}
 \item $G\subset \DD\times \CC^{n-1}$;
 \item for any neighborhood $U\subset \CC^{n-1}$ of $0$ there is a $\delta<1$ such that $G\setminus (\overline\DD \times U)$ is contained in $\delta \DD \times \CC^{n-1}$.
\end{itemize}

Now we shall make use of the following simple observation.
\begin{lemma}\label{lem: elips}
 Let $G$ be a domain contained in $\DD\times \CC^{n-1}$ such that $T:= \partial\DD\times \{0\}$
is contained in $\overline G.$ Assume moreover that the boundary of $G$ is smooth and strongly linearly convex in a neighborhood of $T$.

 Then there is a constant $A>0$ such that $|z_1|+A||z'||^2<1$ whenever $z=(z_1,z')\in G$ is close to $T$.
\end{lemma}

Postponing the proof of the lemma stated above observe that applying it and making use of properties of $G$ we see that there is an $A>0$ such that \begin{equation}\label{eq: A}|z_1|^2+A|z_j|^2\leq |z_1|+A||z'||^2<1\quad \text{for}\ z=(z_1,z')\in G,\; j=2,\ldots,n.\end{equation} Then $G_{A,j}(z):=\frac{z_1}{\sqrt{1-A z_j^2}}$ is a well defined function on $G$ attaining its values in the unit disc, where $j=2,\ldots,n.$ It is clear that $G_{A,j}\circ \Phi$ are left inverses to $f$ provided that $A$ is small enough (or $A=0$).
\end{proof}

\begin{proof}[Proof of Lemma~\ref{lem: elips}]
 Let $r$ be a smooth defining function for $\partial G$ in a neighborhood of $T$.
%Losing no generality we may assume that $r$ is strictly plurisubharmonic in a neighborhood of $T$. %(a standard argument - it suffices to consider $e^{Mr}$ with big $M>0$).
 The strong linear convexity of $\partial G$ implies that a function $z'\mapsto r(z_0,z')$
is strongly convex in a neighborhood of $0$ for any $z_0\in \TT$ (of course $\{z_0\}\times \mathbb C^{n-1}$
is a complex affine space tangent to $\partial G$ at $(z_0,0)$).
 Therefore, there is an open neighborhood $U\subset \CC^{n-1}$ of $0$
and there is an $\alpha>0$ such that \begin{equation}\label{eq: c1} r(z_0,z')\geq \alpha ||z'||^2\quad \text{for any}\quad (z_0,z') \in \TT
\times U.\end{equation} On the other hand, it follows from the smoothness of $r$ that
\begin{equation}\label{eq: c2} r(z_1,0)>C_1 (|z_1|-1)\quad \text{for any $z_1\in\DD$ sufficiently close to $\partial\DD$},
\end{equation}
for some uniform constant $C_1>0.$

 For $z_1\in \DD$, $z_1\neq 0,$ let $\tilde z(z_1):=z_1/|z_1|.$ %Since $\frac{\partial r}{\partial z'}(z_0,0)=0$ for $z_0\in \TT$,
Making use of the inequalities \eqref{eq: c1} and \eqref{eq: c2} one may see that some simple analysis gives constants $\alpha,\beta>0$
such that $$r(z_1,z')\geq -\alpha |z_1-\tilde z(z_1)|+\beta ||z'||^2$$ for any $z_1\in \DD$ close to $\partial\DD$
and $z'$ close to $0.$
From this we immediately deduce that  there are positive constants $A$ and $C$ such that $$r(z)\geq C(|z_1|+A ||z'||^2-1),$$
whenever $z\in G$ is sufficiently close to $T.$ This easily implies the assertion.
\end{proof}

\section{Uniqueness of left inverses in convex domains}\label{section:uniqueness}

\begin{theorem}\label{theorem:uniqueness}
Let $D$ be a convex domain in $\CC^n$ ($n\geq 2$) and let $f^1,\ldots,f^{n}:\DD\to D$ be complex geodesics
such that $f^1(0)=\ldots=f^{n}(0)=w$ and $f^1(\sigma)=\ldots=f^{n}(\sigma)=z$
for some $\sigma\in\DD\setminus\{0\}$. Assume additionally
that for some $\lambda\in\DD$ the system of vectors
\begin{equation}\label{eq:independence}
\{f^1(\lambda)-f^n(\lambda),\ldots,f^{n-1}(\lambda)-f^n(\lambda)\}
\end{equation}
is linearly independent
(equivalently, the convex hull of the set $\{f^1(\lambda),\ldots,f^n(\lambda)\}$ is $(n-1)$-dimensional).
Then there is only one left inverse passing through points $w$ and $z$, i.e. a holomorphic function
$F:D\to\DD$ such that $F\circ f=\id_{\DD}$ for some (equivalently, any) complex geodesic $f:\DD\to D$ such that
$f(0)=w$, $f(\sigma)=z$.
\end{theorem}
\begin{proof}
Let $\emptyset\neq U\subset\subset\DD$ be such that for any $\lambda\in \overline U$ the property (\ref{eq:independence}) is satisfied.
For $\lambda\in\overline U$ define
\begin{multline}
g_{\lambda}(\lambda_1,\ldots,\lambda_{n-1}):=\\
\lambda_1f^1(\lambda)+\ldots+\lambda_{n-1}f^{n-1}(\lambda)+
(1-\lambda_1-\ldots-\lambda_{n-1})f^n(\lambda)=\\
f^n(\lambda)+\lambda_1(f^1(\lambda)-f^n(\lambda))+\ldots+
\lambda_{n-1}(f^{n-1}(\lambda)-f^n(\lambda)).
\end{multline}
Put $T:=\{t=(t_1,\ldots,t_{n-1}):0\leq t_j\leq 1,t_1+\ldots+t_{n-1}\leq 1\})\subset \CC^{n-1}$. Convexity of $D$ implies that $g_{\lambda}(T)\subset D$, $\lambda\in\overline U$.
Let $T\subset\Omega\subset\CC^{n-1}$ be an open connected set such that $g_{\lambda}(\Omega)\subset D$, $\lambda\in\overline U$.
We easily see that  for any
$(t_1,\ldots,t_{n-1})\in T$ the function $\DD\owns \lambda\mapsto g_{\lambda}(t_1,\ldots,t_{n-1})\in D$ is a complex geodesic in $D$ for $(w,z)=(g_{0}(t_1,\ldots,t_{n-1}),g_{\sigma}(t_1,\ldots,t_{n-1}))$. Let $F:D\to\DD$ be any
left inverse passing through $w$ and $z$. Then,
$F(g_{\lambda}(t_1,\ldots,t_{n-1}))=\lambda$, $\lambda\in\DD$, $(t_1,\ldots,t_{n-1})\in T$.
The identity principle for holomorphic functions
implies that
\begin{equation}\label{eq:left-inverse}
F(g_{\lambda}(\lambda_1,\ldots,\lambda_{n-1}))=\lambda,\; \lambda\in\overline U,\;(\lambda_1,\ldots,\lambda_{n-1})\in\Omega.
\end{equation}
Note that the mapping
\begin{equation}
\Phi:U\times\Omega\owns(\lambda,\lambda_1,\ldots,\lambda_{n-1})\mapsto g_{\lambda}(\lambda_1,\ldots,\lambda_{n-1})\in
D\end{equation}
is holomorphic. We claim that the mapping $\Phi$ is injective. In fact, the equality $\Phi(\mu,\mu_1,\ldots,\mu_{n-1}) =
\Phi(\la,\la_1,\ldots,\la_{n-1})$ implies, in view of (\ref{eq:left-inverse}), that $\lambda=\mu$. Consequently,
\begin{multline}
\lambda_1(f^1(\lambda)-f^n(\lambda))+\ldots+
\lambda_{n-1}(f^{n-1}(\lambda)-f^n(\lambda))=\\
\mu_1(f^1(\la)-f^n(\la))+\ldots+
\mu_{n-1}(f^{n-1}(\la)-f^n(\la)).
\end{multline}
Now the linear independence (property (\ref{eq:independence})) shows that $\la_j=\mu_j$, $j=1,\ldots,n-1$.

Consequently, $\Phi(U\times\Omega)$ is a non-empty open subset of $D$ on which the left inverse $F$ is uniquely determined.
The identity principle for holomorphic functions finishes the proof.
\end{proof}

\begin{rem}
 Note that the assumption of Theorem~\ref{theorem:uniqueness} in the case $n=2$
means that there are two different complex geodesics
passing through $w$ and $z$. In higher dimension we need the existence of 'many more'
complex geodesics which would guarantee the uniqueness of left inverses.
\end{rem}

\begin{rem}
The proof  Theorem~\ref{theorem:uniqueness} relies very much upon the convexity of $D$. It is natural to ask the question whether the same result holds without that assumption.
\end{rem}

The above result suggests that the uniqueness of geodesics may imply the non-uniqueness of left inverses. As we shall see
from the proposition below this is in general not the case.

\begin{prop}\label{prop:example}
 Let $D$ be a pseudoconvex complete Reinhardt domain in $\CC^2$.
Assume that the point $p_0:=(1,0)$ lies in the topological boundary of $D$ and that the boundary of $D$
is $\mathcal C^{\infty}$ smooth in a neighborhood of $p_0$. Let $f$ denote a complex geodesic $\DD\owns\la\mapsto (\la,0)\in D$.

Then $f$ has a unique left inverse in $D$ if and only if the boundary of $D$ is of infinite type at the point $p_0$.
\end{prop}

\begin{rem}
It follows from \cite{Din1989} that in the class of bounded strictly convex domains (i.e. bounded convex domains admitting no segments in the boundary) complex goedesics are always unique. As we know from
Theorem~\ref{theorem:strongly} in the case of strong convexity
left inverses are never uniquely determined. This will not be the case if the assumption of strong convexity
were relaxed with a geometric assumption of strict convexity.
Actually, using Proposition~\ref{prop:example} we are able to construct a strictly convex
and smooth Reinhardt domain such that the only left inverse to $\DD\owns \la\mapsto (\la, 0)\in D$
is the projection onto the first variable
(take for example a smooth Reinhardt domain being additionally strictly convex
whose defining function is $|z|+e^{-|w|^{-2}}-1$ in a neighborhood of $p_0$).

%On the other hand recall that complex geodesics in strongly convex bounded domains are always uniquely determined
%(see {\bf trzeba sprawdzic w ksiazce Jar-Pfl1993 lub w pracy Jar-Pfl-Zei o geodezyjnych w elipsoidach...} ).
\end{rem}

\begin{proof}
Assume first that $\partial D$ is of finite type at $p_0$. Let $\rho$ be a defining function of $\partial D$ in a neighborhood of $p_0.$ Write $z=t e^{i\theta},$ $w=s e^{i\eta}.$ We may always assume that $$\rho(z,w)=\rho(t,s)= t+ r(s) -1$$ for $(t,s)$ in a neighborhood of $(1,0)$, where $r(s)= a s^k + O(s^{k+1})$, $a>0$, $k\in \mathbb N$. Therefore, using basic properties of pseudoconvex Reinhardt domains we find that there is $b>0$ such $D\subset \{(z,w)\in \CC^2:\ |z|+ b |w|^k<1\}.$ Then, $$F_\be(z,w)=\frac{z}{1- \be w^k},\quad (z,w)\in D$$ is a left inverse of $f$ for any $0\leq \be< b$.

Now assume that $\partial D$ is of infinite type at $p_0$ and take any left inverse $F:D\to \DD$ of $f$.
Expanding $F$ in a series and making use of the equality $F\circ f=\id_{\DD}$
we get that $$F(z,w)=z+zwP_1(z,w)+w^kQ(z,w),$$ for some positive integer $k$, holomorphic mappings $P_1$ and $Q$
on $D$ such that either $Q\equiv 0$ or $Q(0,0)\neq 0$. Observe that the second possibility cannot occur.
Actually, suppose that $Q(0,0)\neq 0.$ Then, consider the function
$\DD\setminus\{0\}\owns\la \mapsto F(\lambda ^k z, \lambda w) \lambda^{-k}$, which extends to a well defined holomorphic function
on $\DD$ with values in $\overline \DD$.
Letting $\la\to 0$ we get that the mapping $(z,w)\mapsto z+ Q(0,0) w$ maps $D$ into the closed unit disc
which gives the inclusion $D\subset\{|z|+|Q(0,0)||w|<1\}$ which contradicts the smoothness of the domain.

Therefore, $$F(z,w)=z(1+w P_1(z,w)).$$ Our aim is to show that $P_1\equiv 0.$ Clearly, this would imply the assertion of the proposition.

Assuming the contrary take $P$ and a positive integer $k$ such that $P_1(z,w)=w^{k-1}P(z,w)$
and $P(\cdot,0)\not\equiv 0$. Take any $z_0\in 2^{-1}\mathbb D$ such that $P(z_0,0)\neq 0$.
Let $\eta\in \RR$ be such that $|P(z_0,0)|=e^{i\eta} P(z_0,0)$.
Considering $e^{i\eta/k} w$ instead of $w$ we may assume that $\eta=0.$
Let $A>0$ be such that $\re P(z_0,w)\geq A$ for $w$ in a neighborhood of $0$,
say $|w|<\eps,$ for some $\eps >0.$ Certainly $\re P(z_0,w)w^k\geq A w^k$ whenever $0\leq w<\eps.$

Making use of the properties of harmonic functions one can see that for any $2^{-1}<|z|<1$ and $|w|<\eps$
there is $\theta=\theta(z,w) \in \RR$ such that $\re P(z e^{i\theta}, w)\geq A$

Thus, for $z\in \DD$ arbitrarily close to the unit circle
and $0\leq w<\eps$ we have the following sequence of inequalities:
\begin{multline}\label{mult}
|F(e^{i\theta(z,w)} z, w)|=|z||1+w^k P(e^{i\theta(z,w)}z,w)|\geq\\ |z| (1+Re (w^k P(e^{i\theta(z,w)}z,w))
\geq |z|(1+ Aw^k).\end{multline}

Making use of \eqref{mult} and the assumptions on $F$ we infer that the inequality $1\geq |z|(1+A|w|^k)$ holds for $(z,w)\in \overline D$
such that $|z|$ is close to $1$ ans $|w|$ is small. This gives us the contradiction.

\end{proof}

\section{Left inverses in the symmetrized bidisc (Nevanlinna-Pick problem for two points
in the symmetrized bidisc)}\label{section:symmetrizedbidisc}

We start this section with recalling a result on a complete description of holomorphic retracts of the diagonal in
the bidisc that will be an important tool in the proofs of uniqueness parts of left inverses in both domains:
the symmetrized bidisc and the tetrablock. The description is a direct consequence of the following result
that we present below in a form that we could apply directly in our paper.

\begin{theorem}\label{theorem:retracts} {\rm (see Example 11.79 in \cite{Agl-McC2002})}
Let $F:\DD\times\DD\to\DD$ be a holomorphic function such that $F(\la,\la)=\la$, $\la\in\DD$. Then there is
an $h$ lying in the closed unit ball of $H^{\infty}(\DD^2)$ and $t\in[0,1]$ such that
\begin{equation}
 F(z_1,z_2)=\frac{tz_1+(1-t)z_2-z_1z_2h(z_1,z_2)}{1-((1-t)z_1+tz_2)h(z_1,z_2)},\; z_1,z_2\in\DD.
\end{equation}
Conversely, any function defined as above maps $\DD^2$ to $\DD$ and satisfies the equality $F(\la,\la)=\la$,
$\la\in\DD$.
\end{theorem}

Recall that $\tilde k_{\GG_2}=c_{\GG_2}$ (see \cite{Agl-You2004}, \cite{Cos2004}). Consequently,
there are complex geodesics passing through arbitrary pair of points.
%Let $f:\DD\to\GG_2$ be a complex geodesic in the symmetrized bidisc.
Below we deal with the uniqueness of the left inverse to complex geodesics in the symmetrized bidisc.
%i.e. existence of a holomorphic mapping $F:\GG_2\to\DD$ such that $F\circ f=\id_{\DD}$.

To solve this problem it is sufficient to consider two cases that are presented in the theorem giving a complete description of
complex geodesics in the symmetrized bidisc that we give below (see \cite{Pfl-Zwo2005}). Recall that in \cite{Agl-You2006} another description of complex geodesics in the symmetrized bidisc was given.

\begin{theorem}\label{theorem:symmetrizedbidisc-geodesics} {\rm (see \cite{Pfl-Zwo2005}).}
A holomorphic mapping $f:\DD\to\GG_2$ is a complex geodesic if and only if it is
(up to an automorphism of the unit disc and automorphism of the symmetrized bidisc) of one of the following two forms:
\begin{equation}\label{blaschke}
f(\lambda)=\pi(B(\sqrt{\lambda}),B(-\sqrt{\lambda})),
\end{equation}
 $\lambda\in\DD$, where $B$ is a non-constant Blaschke
product of degree one or two with $B(0)=0$;

\begin{equation}\label{automorphism}
 f(\lambda)=\pi(\lambda,a(\lambda)),
\end{equation} $\lambda\in \DD$,
 where $a$ is an automorphism of $\DD$ such that $a(\lambda)\neq\lambda$, $\lambda\in\DD$.

Moreover, the complex geodesics in the symmetrized bidisc are unique (up to automorphisms of the unit disc).
\end{theorem}

In our considerations (because of the form of automorphisms of $\GG_2$)
it is sufficient
to take in the above theorem
\begin{equation}\label{blaschke-special}
B(\lambda)=\lambda\frac{\lambda-\alpha}{1-\bar\alpha\lambda}\text{
for some $\alpha\in\DD\cup\{1\}$}
\end{equation}
 -- we shall assume that below.

\begin{rem}\label{rem:leftG2}Recall that for each complex geodesic one of the possible inverses is (up to an automorphism of the unit disc) a function
$\Psi_{\omega}(s,p)=\frac{2p-\omega s}{2-\bar\omega s}$, where $\omega\in\overline\DD$. It is elementary to observe that no
two such functions are equal up to automorphisms of $\DD$, i.e. there are no $\omega_1,\omega_2\in\overline\DD$, $a\in\Aut\DD$ such that
$\Psi_{\omega_1}\equiv a\circ\Psi_{\omega_2}$.

Moreover, in all cases except for the case (\ref{blaschke}) and $\alpha=1$ the number $\omega$ is from $\partial\DD$.

Note that in the case (\ref{blaschke}) and $\alpha=0$ any function $-\bar\omega\Psi_{\omega}$, $\omega\in\partial\DD$
is a left inverse to $f$.

In the case (\ref{blaschke}) and $\alpha=1$ any function $-\Psi_{\omega}$, $\omega\in\overline\DD$
is a left inverse to $f$.
\end{rem}

Let us also denote the special analytic set lying in $\GG_2$ the so called {\it royal variety of $\GG_2$}:
\begin{equation}
 \mathcal T:=\{(s,p)\in\GG_2:s^2=4p\}=\{(2\la,\la^2):\la\in\DD\}.
\end{equation}

\begin{theorem}\label{theorem:symmetrizedbidisc-leftinverse} Let $f:\DD\to\GG_2$ be a complex geodesic of one of the forms
 from Theorem~\ref{theorem:symmetrizedbidisc-geodesics}.

If $f$ is of the form as in (\ref{blaschke}) with $B$ satisfying (\ref{blaschke-special}) then $f$ has a unique left inverse if and only if
$\alpha\in\DD\setminus\{0\}$.

If $f$ is of the form as in (\ref{automorphism}) then $f$ has a unique left inverse if and only if the function $\la\mapsto a(\lambda)-\lambda$
has a double root on the unit circle $\partial\DD$. Moreover, both cases do hold.
\end{theorem}

\begin{rem}
 Note that the above theorem is a generalization of Theorem 1.6 in \cite{Agl-You2006}
where the problem of uniqueness of left inverses (but restricted to functions $\Psi_{\omega}$) is settled.
\end{rem}

\begin{proof} In view of the considerations before the theorem
in the first case we need to consider only the case (\ref{blaschke}) with $\alpha\in\DD\setminus\{0\}$.

At first note that there is an (isometric) identification between the space
$H^{\infty}(\GG_2)$ and $H^{\infty}_s(\DD^2):=\{F\in H^{\infty}(\DD^2):F(z_1,z_2)=F(z_2,z_1)\}$ given by the formula
\begin{equation}
 H^{\infty}_s(\DD^2)\owns F\mapsto \{\GG_2\owns (s,p)=\pi(z_1,z_2)\mapsto F(z_1,z_2)\in\CC\}\in H^{\infty}(\GG_2).
\end{equation}

Fix $\sigma\in\DD\setminus\{0\}$. Consider now the following three point Nevanlinna-Pick problem in the bidisc.

\begin{equation}\label{pick-bidisc}
  (0,0)\to 0,\; (B(\sigma),B(-\sigma))\to\sigma^2,\;
  (B(-\sigma),B(\sigma))\to\sigma^2.
\end{equation}

We already know that there is a holomorphic $F:\GG_2\to\DD$ such that $F(f(\lambda))=\lambda$.  Consequently, the function
$F\circ\pi$ is the Nevanlinna-Pick solution for the above problem. We claim that the solution is unique
which would imply the uniqueness of the left inverse to $f$.

We make use of the results from \cite{Agl-McC2002} (Theorem 12.13). According to that result the problem (\ref{pick-bidisc}) will be unique, which shall show that the corresponding Nevanlinna-Pick problem for the symmetrized bidisc
is unique, if the problem (\ref{pick-bidisc}) is
\begin{itemize}
 \item extremal,
 \item non-degenerate, i.e. none of three 2-point Pick subproblems of problem \eqref{pick-bidisc} is extremal,
 \item strictly two-dimensional.
\end{itemize}
The extremality of the problem is a consequence of the fact that if there were a solution $\tilde G$ of the problem (\ref{pick-bidisc})
with the norm smaller than one then the (symmetrization) function $G$ given by the formula
\begin{equation}
 G(z):=(\tilde G(z_1,z_2)+\tilde G(z_2,z_1))/2
\end{equation}
 would be symmetric with the norm smaller than $||G||<1$. This would yield
a holomorphic function (denoted by the same letter) $G$ defined on $\GG_2$ with the norm smaller than one and such that
$G(0)=0$ and $G(\pi(B(\sigma),B(-\sigma))=\sigma^2$. That would mean that the function $f$ were not a geodesic - contradiction.

The fact that the problem (\ref{pick-bidisc}) is non-degenerate will follow from the inequality
\begin{equation}
 |\sigma^2|<\max\{|B(\sigma)|,|B(-\sigma)|\}
\end{equation}
or equivalently
\begin{equation}
 |\sigma|<\max\{\left|\frac{\sigma-\alpha}{1-\bar\alpha\sigma}\right|,\left|\frac{\sigma+\alpha}{1-\bar\alpha\sigma}\right|\}.
\end{equation}
Note that the Poincar\'e closed disc with center at $\sigma$ and the radius $p(0,\sigma)$ contains at most one of the points:
$\alpha$ and $-\alpha$ for $\alpha\in\DD\setminus\{0\}$. This fact implies the last inequality and completes the proof of non-degeneracy.

To show the fact that the problem is strictly two-dimensional let us note
that otherwise we would have the existence of a holomorphic $\varphi:\DD\to\DD$
with $\varphi(0)=0$, $\varphi(B(\sigma))=\sigma^2$, $\varphi(B(-\sigma))=\sigma^2$.
Write $\varphi(B(\lambda))=\lambda\psi(\lambda)$
where $\psi:\DD\to\overline\DD$. Then $\psi(\sigma)=\sigma$ and $\psi(-\sigma)=-\sigma$. Now the Schwarz-Pick Lemma implies
that $\psi=id_{\DD}$ and then $\varphi(B(\la))=\la^2$ so
\begin{equation}
0=\varphi(B(0))=\varphi(B(\alpha))=\alpha^2
\end{equation}
-- a contradiction.

\vskip1cm

Now we consider the case (\ref{automorphism}).
Let $a(\lambda)=\tau\frac{\lambda-\alpha}{1-\bar\alpha\lambda}$ where $|\tau|=1$. Recall that the condition
$a(\lambda)\neq\lambda$, $\lambda\in\DD$ is equivalent to the inequality
 \begin{equation}\label{inequality}
  |1-\tau|\leq 2|\alpha|.
 \end{equation}
It follows from the second part of the proof of Theorem~3 in \cite{Pfl-Zwo2005} that $\Psi_{\omega}$ is (up to an automorphism) a left inverse
iff
\begin{equation}\label{description}
 \frac{(1+\tau\alpha\bar\omega)^2}{\tau}>0.
\end{equation}
It was proven in \cite{Pfl-Zwo2005} that for the fixed $\tau$ and $\alpha$
there is an $\omega$ satisfying the above inequality if and only if $|1-\tau|\leq 2|\alpha|$. Therefore,
simple geometric observation shows that in the case the inequality
in (\ref{inequality}) is sharp there will be two $\omega$'s from $\partial\DD$ for which the inequality in (\ref{description})
holds. This implies that in the case (\ref{automorphism}) with the sharp inequality in (\ref{inequality}) there is no uniqueness of
left inverses.

We are left with the case when the inequality (\ref{inequality}) becomes equality, i.e. we assume that
$|1-\tau|=2|\alpha|$ which is equivalent to the fact that the equation $a(\la)-\la$ has a double root on the unit circle.
We claim that in this situation there is only one left inverse.

Let $F$ be a left inverse to $f$, $f(\la)=\pi(\la, a(\la))$.
Then $G=F\circ \pi:\DD^2 \to \DD$ is a left inverse for the geodesic (in $\DD^2$): $\DD\owns\la \mapsto (\la, a(\la))\in\DD^2$.
It is also a symmetric function.

%Making use of Theorem~\ref{theorem:retracts}
Making use of a description of all solutions of the two-point Nevanlinna-Pick problem
\begin{equation}
 (0,0)\mapsto 0,\quad (1/2,1/2)\mapsto 1/2
\end{equation}
in the bidisc (or equivalently the descriptions of the retracts of the diagonal in the bidisc) we get in view
of  Theorem~\ref{theorem:retracts} that

%(see Example 11.79 from \cite{Agl-McC2002}) we get that there is a $t\in [0,1]$ such that
%$$
%G(z_1,z_2)=tz_1+(1-t)b(z_2)+\frac{t(1-t)(z_1-b(z_2))^2h(z_1,z_2)}{1- ((1-t)z_1+ t b(z_2))h(z_1,z_2))},\;z_1,z_2\in\DD,
%$$ where $b=a^{-1}$ (note that $b(\la)-\la$ has a double root on $\TT$ as well)
%and $h$ lies in the closed unit ball of $H^{\infty}(\DD^2)$.
%Elementary computations give:
\begin{equation}\label{eq:unique-left-inverse}
G(z_1,z_2)=\frac{tz_1+(1-t)b(z_2)-z_1 b(z_2)h(z_1,z_2)}{1- ((1-t)z_1+ tb(z_2))h(z_1,z_2))}
\end{equation}
where $h$ lies in the closed unit ball of $H^{\infty}(\DD^2)$, $t\in[0,1]$ and $b:=a^{-1}$.

Let $b(\la)=\sigma\frac{\la - \beta}{1-\la \bar\beta},$
where $\sigma\in \TT$ and $\beta\in \DD_*$. Note that the relation $|1-\sigma|=2|\beta|$ is satisfied (consequently,
$\sigma$ cannot be equal to $\pm 1$). Observe that $G(b(\la),\la))=b(\la)$, in particular, making use of the symmetry of $G$
we find that $\la\mapsto G(\la, b(\la))$ vanishes at $\beta$.
In other words $t\beta +(1-t) b(0)-\beta b(0)h(\beta,0)=0$.
Since $\beta\neq 0$ we easily find that $t-(1-t)\sigma+\sigma \beta h(\beta,0)=0$,
i.e. \begin{equation}
\label{eq: h} h(\beta,0)=\frac{t - (1-t)\sigma}{-\sigma \beta}.
\end{equation}
The modulus of the term on the right side of the above equation attains its (unique) minimum when $t=1/2.$
Therefore, $$|h(\beta, 0)|\geq \frac{|1-\sigma|}{2|\beta|}=1.$$
The maximum principle implies that $h$ is constant.
As we have already shown, $t$ is uniquely determined ($t=1/2$) so (the constant function) $h$ is uniquely determined
as well by the equation \eqref{eq: h}.
\end{proof}

\begin{rem}\label{remark:uniqueness-of-t}
It follows from the proof of the previous theorem that the uniquely determined function
$G$ in the equality (\ref{eq:unique-left-inverse}) is such that $t=1/2$ and $h$ is a constant from $\TT$ equal to $\frac{\sigma-1}{2\sigma\beta}$ (note that $h$ is a unique point in the unit circle such that $a(\bar h) = \bar h$).
\end{rem}
Using this observation we get the following result which plays an important role in the next section.
\begin{lemma}\label{remg2} Let $\phi$ be complex geodesic in the symmetrized bidisc of the form $\phi:\DD\owns \la\mapsto \pi(a(\la), b(\la))\in\GG_2$, where $a$ and $b$ are M\"obius maps. Assume that the equation $a(\la)=b(\la)$, $\la\in \overline \DD$, has one double root lying on the unit circle and let $F:\GG_2\to \DD$. Then the only left inverse of $\phi$ is given by the formula
$$F(\pi(z_1, z_2)) = \frac{\frac12 \left(a^{-1}(z_1)+ b^{-1}(z_2) \right) - a^{-1}(z_1) b^{-1}(z_2) h}{1- \frac12\left(a^{-1}(z_1) + b^{-1}(z_2) \right) h},\quad z_1,z_2\in \DD,$$ where $h\in \TT$ solves the equation $a(\bar h) = b(\bar h)$.
\end{lemma}
\begin{proof}
The mapping $$\pi_a:\GG_2\ni \pi(\la_1, \la_2)\mapsto \pi(a(\la_1), a(\la_2))\in \GG_2$$ is an automorphism of $\GG_2$ and $F\circ \pi_a$ is a left inverse of the geodesic of the form $\la\mapsto \pi(\la, a^{-1}(b(\la)))$. Therefore, the assertion follows from Remark~\ref{remark:uniqueness-of-t}.
\end{proof}

\section{Left inverses in the tetrablock (Nevanlinna-Pick problem for two points
in the tetrablock)}\label{section:tetrablock}

Recall that all invariant functions coincide on $\EE$ (see \cite{Edi-Kos-Zwo2012}).
In particular, complex geodesics passing through any pair of points in the tetrablock do exist.
Moreover, a special role in the study of the geometry of $\EE$ is played by the set
$\Sigma:=\{z\in\EE:z_1z_2=z_3\}$ called the \textit{royal variety of $\EE$}.

We start with recalling basic properties of $\EE$.
First note that $\sigma:\EE\owns z\mapsto(z_2,z_1,z_3)\in\EE$ is an automorphism of $\EE$.

For $a,b\in\DD$
and  $\omega_1,\omega_2\in\TT$ we put
\begin{multline}\label{autE} \Psi_{a,b,\omega_1,\omega_2}(z_1,z_2,z_3)=\\
\left(\omega_1\frac{z_1+a+\overline {b}z_3+a\bar b z_2}{1+\bar az_1+\bar bz_2+\overline{ab}z_3}, \omega_2\frac{z_2+b+\overline az_3+b\overline
az_1}{1+\bar az_1+\bar bz_2+\overline{ab}z_3},
\omega_1\omega_2\frac{z_3+az_2+bz_1+ab}{1+\bar az_1+\bar
bz_2+\overline{ab}z_3} \right),
\end{multline}
$z=(z_1,z_2, z_3)\in \EE$. Recall that the group of automorphism of $\EE$ is given by $$\Aut(\EE)=\{\Psi_{a,b,\omega_1,\omega_2},
\Psi_{a,b.\omega_1,\omega_2}\circ\sigma:
a,b\in\DD, \omega_1,\omega_2\in\TT\}$$ (see \cite{You2008}, see also \cite{Kos2011}).

Similarly as in the case of the symmetrized bidisc for any pair of points there exists
a complex geodesic $\varphi:\DD\to \EE$ passing through them satisfying one the following three conditions:
\begin{enumerate}\label{enumarate:(i)}
\item $\varphi$ lies entirely in the royal variety $\Sigma$,
 \item $\varphi$ intersects $\Sigma$ exactly at one point,
 \item $\varphi$ omits $\Sigma$.
\end{enumerate}

It is well known (see \cite{You2008}) that $\Aut(\EE)$ acts transitively on $\Sigma$. Therefore, if $\varphi$ is a complex geodesic of $\EE$ such that $\varphi(\DD)\cap \Sigma\neq\emptyset$ we may always assume that $\varphi(0)=0$.

In the case when $\va:\DD\to \EE$ is a complex geodesic of $\EE$ such that $\va(\DD) \cap \Sigma = \emptyset$ making use of the description of the group of automorphisms of the tetrablock we may assume that $\va(0)=(0,0,-\be^2)$ for some $\be\in (0,1).$

The following theorem summarizes the description of complex geodesics in the tetrablock obtained in \cite{Abo-Whi-You2007}, \cite{Edi-Zwo2009} and \cite{Edi-Kos-Zwo2012}.
\begin{theorem}\label{geo-E} {\rm (see \cite{Abo-Whi-You2007}, \cite{Edi-Zwo2009}, \cite{Edi-Kos-Zwo2012})}
Let $\varphi:\DD\to \EE$ be a complex geodesic in the tetrablock.

If $\va(0)=0$, then there are $\omega_1,\omega_2\in \TT$, a number $C\in [0,1]$ and a holomorphic mapping $\psi:\DD\to \overline \DD$, $\psi(0)=-C$ such that
\begin{equation}\label{form0} \varphi(\la)=\left(\omega_1 \frac{\psi(\la)+C}{1+C}, \omega_2 \la \frac{1+C \psi(\la)}{1+C}, \omega_1 \omega_2 \la \psi (\la)\right),\quad \la\in \DD.
\end{equation}

If $\va(\DD)\cap \Sigma=\emptyset$ and $\va(0)=(0,0,-\be^2)$ for some $\be>0$, then there exist
$a,b,c,d\in\overline\DD$ with $|a|^2+|b|^2=|c|^2+|d|^2=1$ and
$a\overline c+b\overline d=0$  and there exists a holomorphic mapping $Z:\DD\to\DD$ satisfying either $Z(\la)=\la$, $\la\in \DD,$ or $|Z(\la)|<|\la|,$ $\la\in \DD\setminus\{0\},$ such that
\begin{equation}\label{formofva}
\va(\lambda)=\left(
\frac{A(\lambda)(1-\beta^2)}{\Delta(\lambda)},
\frac{C(\lambda)(1-\beta^2)}{\Delta(\lambda)},
\frac{A(\lambda)C(\lambda)-(B(\lambda)+\beta)^2}{\Delta(\lambda)}
\right),
\end{equation}
where $A(\lambda)=a^2\lambda+b^2 Z(\lambda)$,
$B(\lambda)=ac\lambda+bd Z(\lambda)$, $C(\lambda)=c^2\la +d^2
Z(\lambda)$, and
$\Delta(\lambda)=(1+\beta
B(\lambda))^2-A(\lambda)C(\lambda)\beta^2$.
\end{theorem}

Observe that if $\va$ is of the form \eqref{form0}, then $\va(\DD)$ lies entirely in the royal variety $\Sigma$ if and only if $C=0.$ In the other case $\va(\DD)$ intersects $\Sigma$ only at $0$.

\begin{rem}\label{rem:interS} {\rm (see \cite{Edi-Kos-Zwo2012})}. Suppose that $\va:\DD\to \EE$ is a complex geodesic of the form \eqref{formofva}.

If $|Z(\la)|<|\la|,$ $\la\in \DD,$ then the condition $\va(\DD)\cap \Sigma=\emptyset$ is equivalent to $$(1+\be^2) |c||d| \leq \be.$$

Moreover, if $Z(\la)=\la,$ $\la\in \DD$, then the condition $\va(\DD)\cap \Sigma=\emptyset$ is always satisfied.

\end{rem}

As shown in \cite{Abo-Whi-You2007} for any complex geodesic $\va$ of $\EE$ intersecting the royal variety there is an $\omega\in \TT$ such that $\Phi_{\omega} \circ f\in \Aut(\DD)$ or $\Phi_{\omega}\circ \sigma \circ f\in \Aut(\DD)$,
where $$\Phi_{\omega}(z)=\frac{\omega z_3 - z_1}{\omega z_2 -1},\quad z\in \EE.$$

On the other hand the family of functions $\Phi_{\omega}$ and $\Phi_{\omega}\circ \sigma$ is not rich enough to replace
in the definition of the Carath\'eodory distance of $\EE$ the family of all bounded holomorphic functions (see \cite{Edi-Zwo2009}). However, we were able to construct functions that we need to add to this family. They were defined in \cite{Edi-Kos-Zwo2012} in a quite implicit way. Below, for the convenience of the reader, we present a formal construction. The functions that we need to add to this family contain the functions $\tilde\Phi_{\omega}$ and $\tilde\Phi_{\omega}\circ \sigma$ that we construct below. Namely,
it follows from the Rouch\'e Theorem that for any $z\in \EE$, $\omega\in\TT$ the equation
\begin{equation}
\Phi_{\omega}(z_1,\lambda z_2,\lambda z_3)=\lambda
\end{equation}
has exactly one solution $\la=:\tilde \Phi_{\omega}(z)$ in the unit disc.
This defines a holomorphic function $\tilde\Phi_{\omega}:\EE\to \DD$ which is a left inverse for some of the geodesics in $\EE$.
After some simple calculations one may find an explicit formula for it: $$\tilde \Phi_{\omega}(z):=\frac{2 z_1}{1+\omega z_3+\sqrt{(1 +\omega z_3)^2 - 4\omega z_1 z_2}}.$$

Consider the mapping
\begin{equation}\label{eq:coverE}
 \Pi:\left(\begin{array}{cc}
  z_1 & a_1 \\
  a_2  & z_2 \\
\end{array}%
\right)
\mapsto(z_1,z_2,z_1z_2-a_1a_2)\end{equation}
Let $\mathcal R_{II}$ denote the \emph{Classical Cartan domain of the second type} in $\CC^2$, i.e. the set of all symmetric two-by-two matrices with the operator norm smaller than one.
Recall that the tetrablock may be equivalently given as the image of $\mathcal R_{II}$ under the mapping $\Pi$ (see \cite{Abo-Whi-You2007}). For the basic properties of classical Cartan domains we refer the reader to \cite{Hua}.

We shall also make use of the connection between the symmetrized bidisc and the tetrablock. Namely, for any $\omega\in\TT$ there is an embedding of $\GG_2$ into $\EE$ given by the formula
\begin{equation}
\GG_2\owns (s,p)\mapsto(\omega s/2,s/2,\omega p)\in\EE
\end{equation}

On the other hand for any $\omega\in\overline\DD$ we have a mapping (see \cite{Bha2012}, see also \cite{Zwo2013})
\begin{equation}
\EE\owns x\mapsto(x_1+\omega x_2,\omega x_3)\in\GG_2.
\end{equation}

Below we present a (complete) solution of the problem of uniqueness of left inverses in the tetrablock.

\begin{theorem} Let $\va:\DD\to \EE$ be a complex geodesic of the form \eqref{form0}. Then  $\va$ has a unique left inverse if and only if
 $C\in (0,1)$ and the function $\psi$ is not an automorphism of $\DD$.

If $\va:\DD\to \EE$ is a complex geodesic of the form \eqref{formofva}, then $\va$ has a unique left inverse if and only if $|Z(\la)|<|\la|,$ $\la\in \DD\setminus\{0\},$ and $|c||d|(1+\be^2) =\be$.
\end{theorem}

The proof of the above theorem is presented below and it is divided into several steps.
\subsection{Geodesics lying entirely in $\Sigma$}
Any complex geodesic $\varphi:\DD\to\EE$ such that $\varphi (\DD)\subset \Sigma$ and $\varphi(0)=0$ is, up to a composition with a linear automorphism of the tetrablock,  of the form
$$\varphi(\la)=(\la,a(\la),\la a(\la)),\quad \la \in \DD,$$ where a holomorphic function $a:\DD\to \DD$ fixes the origin.
Then both $\Phi_1$ and the projection on the first variable are left inverses for $\varphi$.

\subsection{Geodesics intersecting $\Sigma$ exactly at one point}\label{ss} Let us focus on the possibility when the complex geodesic
$\va:\DD\to \EE$ satisfies $\# (\varphi (\DD)\cap \Sigma)=1$ and $\varphi(0)=0$. Then $\varphi$ is of the form \eqref{form0}, where $C\in (0,1]$ (recall that $C=0$ cannot occur in this case). If $C=1$ then $\varphi(\lambda)=(0,0,-\omega_1\omega_2\la)$ and then
$\eta_1\Psi_1$ and $\eta_2\Psi_1\circ\sigma$ for some $|\eta_j|=1$ are left inverses.

Assume now that $C\in(0,1)$. Then $\psi$ is a holomorphic selfmapping of the unit disc. Rewriting the formula for $\varphi$ (composing $\psi$ with a M\"obius map) we get that
\begin{equation}\label{form1} \varphi(\la)=\left( \omega_1 \frac{(1-C) \psi(\la)}{1-C\psi(\la)},\omega_2 \la \frac{1-C}{1- C\psi(\la)},\omega_1 \omega_2 \la \frac{\psi(\la) - C}{1-C\psi(\la)}\right),\quad \la\in \DD,
\end{equation}
 where $C\in (0,1)$ and $\psi\in \mathcal O(\DD, \DD)$ is such that $\psi(0)=0.$
Composing $\varphi$ with a linear automorphism of $\EE$ we may assume that $\omega_1=\omega_2=1.$
Write $\psi(\la)=\la Z(\la),$ where
\begin{enumerate}
 \item either $Z$ is a holomorphic self mapping of the unit disc,
 \item or $Z(\lambda)=e^{i\theta}$, $\la\in \DD$, for some $\theta\in \RR.$

\end{enumerate}

Let us assume that the first possibility holds, i.e. $|\psi(\la)|<|\la|$, $\la \in \DD\setminus \{0\}$. We shall prove that in this case the left inverse is uniquely determined.

Let $F:\EE\to \DD$ be a left inverse of $\va$. Put
$$\tilde \va(\lambda,\nu):=\left(\frac{(1-C) \nu}{1-C \nu}, \la \frac{1-C}{1- C \nu}, \la \frac{\nu - C}{1-C \nu}\right),\quad \la,\nu\in \DD.
$$
Note that $\tilde\va(\DD^2)\subset\EE$.
Observe that $F\circ \tilde \va$ does not depend on the second variable. Actually, it suffices to note that $F\circ \tilde \va$ is a left inverse for
$\DD\owns \lambda\mapsto (\lambda, \psi(\lambda))\in\DD^2$ and make use of the uniqueness of left inverses in $\DD^2$
following from the non-uniqueness of geodesics (we may use Theorem~\ref{theorem:uniqueness} and remark after that theorem, too).
An immediate consequence of this observation is that $$F(\tilde \va(\lambda, \omega \lambda))=\lambda\quad \text{for any}\ \omega\in \overline\DD.$$

Let us write $\tilde \va=(\tilde \va_1, \tilde \va_2, \tilde \va_3).$ Note that
$$\tilde \va_1(\lambda, \omega\lambda)=\omega \tilde \va_2(\lambda, \omega\lambda),\quad \omega\in \overline \DD,\ \lambda\in \overline\DD.$$

Embedding $\GG_2$ into $\EE$
one may infer that for any $\omega\in\TT$ the mapping
\begin{equation}
 f_{\omega}:\DD\owns\la\mapsto (2\tilde \va_2(\la,\omega\la),\omega^{-1} \tilde \va_3(\la,\omega\la))\in\GG_2
\end{equation}
is a geodesic in the symmetrized bidisc intersecting its royal variety
$\mathcal T$ only at $(0,0).$

Making use of the results of Section~\ref{section:symmetrizedbidisc} (uniqueness of left inverses in the symmetrized bidisc - Theorem~\ref{theorem:symmetrizedbidisc-leftinverse}) we find that the (unique) left inverse of $f_{\omega}$ is given by the formula
\begin{equation}
 \GG_2\owns (s,p)\mapsto F(\omega s/2,s/2, \omega p).
\end{equation}\label{FG}
It follows from the description of left inverses in $\GG_2$ (see Remark~\ref{rem:leftG2}) that
\begin{equation}\label{FEQ} F(\omega s/2, s/2, \omega p)=b(\omega)\frac{2a(\omega)p - s}{2-a(\omega)s}\quad (s,p)\in \GG_2,\ \omega\in \TT,
\end{equation}
where $a(\omega),b(\omega)\in \TT$. Putting $p=0$ in \eqref{FEQ} we find that
\begin{equation}\label{fG} (s,\omega)\mapsto \frac{b(\omega)}{2-a(\omega) s}\end{equation} extends holomorphically to an open neighborhood of $\{0\}\times \overline{\DD}.$ Putting $s=0$ in \eqref{fG} we find that $b$ extends holomorphically to $\overline \DD.$ Differentiating \eqref{fG} with respect to $s$ and putting $s=0$ we get that $a^kb$ extends holomorphically to $\overline\DD$ for any $k\in \mathbb N.$ From this we easily deduce that $a$ extends holomorphically to $\overline\DD,$ as well.

Since $F(\omega\tilde \va_2(\la,\omega\la), \tilde \va_2(\la, \omega\la), \tilde \va_3(\la,\omega\la)) = \la$   for any $\la\in \DD$, $\omega\in\TT$
we find that $b$ may vanish only at $0$ (and then the multiplicity at $0$ is at most $1$). Therefore, either $b(w)=e^{i\theta} w$ or $b(w)=e^{i\theta}.$ Suppose that the first possibility holds. Then
$$F(\omega s/2,s/2, \omega p)= e^{i\theta} \frac{2a(\omega)\omega p - s\omega }{2-a(\omega)s},\quad (s,p)\in\GG_2,\omega\in\overline\DD.
$$
In particular, putting $s=0$ we find that $F(0,0,\oo p)=e^{i\theta} a(\oo) \oo p$ for $\oo,p\in \DD,$ whence $a$ is a unimodular constant and $$F(z)=e^{i\theta} \frac{ a z_3 - z_1}{1-a z_2},\quad z\in \EE.$$ Since $F(\tilde \va (\la,\nu))=\la,$ $\la,\nu \in \DD$, putting $\la=0$ and making use of the above equality we get a contradiction.

Therefore, $b=e^{i\theta}.$ Putting $s=0$ in \eqref{FEQ} we obtain the equality $$F(0,0,\omega p)=e^{i\theta} a(\omega) p,$$ which implies that $a(\omega) = e^{i\eta} \omega.$
This, together with the equality $F\circ \tilde \va(\cdot,\nu)=\id_{\DD}$ easily shows that $e^{i\eta}$ is uniquely defined which shows the uniqueness of $F$.

Finally, note that if the case (2) holds, i.e. $Z(\la)=e^{i\theta}$, $\la\in \DD,$ then $\varphi_1 (\la)=\omega \va_2(\la),$ $\la\in \DD$, for some $\omega\in \TT.$ Thus $\eta_1 \Phi_{\omega_1}$ and $\eta_2 \Phi_{\omega_2}\circ \sigma$ are the left inverses of $\va$, where $\eta_i,\omega_i\in \TT$ are appropriately chosen.

\subsection{Geodesiscs omiting $\Sigma$}

Let us consider the case when $\va:\DD\to \EE$ is a complex geodesic of $\EE$ such that $\va(\DD) \cap \Sigma = \emptyset$ and $\va(0)=(0,0,-\be^2)$ for some $\be\in (0,1).$ Then it follows from Theorem~\ref{geo-E} that there exist
$a,b,c,d\in\overline\DD$ with $|a|^2+|b|^2=|c|^2+|d|^2=1$ and
$a\overline c+b\overline d=0$  and there exists a holomorphic mapping $Z:\DD\to\DD$  such that $Z(\la)=\la$, $\la\in \DD,$ or $|Z(\la)|<|\la|,$ $\la\in \DD\setminus\{0\},$ such that $\va$ is of the form \eqref{formofva}.

If $Z(\la)=\la$, $\la \in \DD$, then $\varphi_1=\omega \varphi_2$ for some $\omega$ in the unit circle. One may check that $\eta_1 \Phi_{\omega_1}$ and $\eta_2 \Phi_{\omega_2}\circ \sigma$ are two different left inverses of $\va$, for some $\eta_i,\omega_i\in \TT$, $i=1,2$.

If $|Z(\la)|<|\la|,$ $\la\in \DD\setminus\{0\}$, then some elementary calculations show that the condition $\va(\DD)\cap \Sigma=\emptyset$ means that $|c||d|(1+\beta^2)\leq\beta$ (see Remark~\ref{rem:interS}). We shall consider two possibilities.

First we focus on the case $|c||d|(1+\beta^2)<\beta.$ In this case we show that left inverses are not uniquely determined.

Applying the results of Section~4 of \cite{Edi-Kos-Zwo2012} (precisely, keeping the notation from that paper, it is shown there that if $f$ given by formula (24) is a complex geodesic, then for some $\tau$ and $\gamma$ the function $g$ defined by formula (29) is a complex geodesic as well) we find that there is a M\"obius map such that one of the mappings $(\va_1, m\va_2, m \va_3)$ or $(m \va_1, \va_2, m\va_3)$ is a complex geodesic of $\EE$ intersecting $\Sigma$ only at a point $a\in \DD$ for which $m(a)=0$. Losing no generality we suppose that the first possibility holds. Denote $b:=\va_1(a)$ and let $\Psi$ be an automorphism of $\EE$ of the form (see \eqref{autE}) $$\Psi(z)=\Psi_{-b,0,1,1}(z)=\left( \frac{z_1 - b}{1- \bar b z_1},  \frac{z_2 - \bar b z_3}{1- \bar b z_1},  \frac{z_3 -\bar b z_2}{1- \bar b z_1}\right),\quad z\in \EE.$$ Note that $\Psi$ maps $(b,0,0)$ to $0$ and observe that $$\Psi\circ \tau_m = \tau_m \circ \Psi,$$ where $\tau_\la (z_1,z_2,z_3)=(z_1,\la z_2, \la z_3),$ $z\in \CC^3,$ $\la\in \CC.$ Therefore, composing, if necessary, the geodesic $\va$ and its left inverse with an automorphism of $\EE$ and with a M\"obius map we may always reduce the problem to the following one: $\va_1(0)=0$ and $\la\mapsto \tau_\la  (\va(\la))= ( \va_1,\la \va_2, \la \va_3)$ is a geodesic in $\EE.$ Now one can use the description of the geodesics of $\EE$ intersecting
$\Sigma$ given in \eqref{form1} to find a formula for $\la\mapsto (\va_1, \la\va_2, \la \va_3)$ which gives the following formula for $\va$
$$
\varphi(\la)=\left(\oo_1 \frac{(1-\be^2)\la}{1-\be^2 \la \tilde Z(\la)}, \oo_2 \frac{(1-\be^2) \tilde Z(\la)}{1-\be^2 \la \tilde Z(\la)}, \oo_1 \oo_2 \frac{\la \tilde Z(\la)-\be^2}{1- \be^2\la \tilde Z(\la)}\right), \quad \la\in \DD,
$$
where $\tilde Z:\DD\to \DD$ is a holomorphic function or $\tilde Z$ is a unimodular constant and $\oo_1, \oo_2\in \TT$. Of course, we may assume that $\oo_1=\oo_2=1$. Note that the case $\tilde Z(\la)=e^{i\theta}$, $\la\in \DD$, for some $\theta\in \RR$ is not possible. Similarly, the fact that $\tilde Z$ is a M\"obius maps is in a contradiction with the assumption $|Z(\la)|<|\la|,$ $\la\in \DD\setminus \{0\}$.

We already know that $\Phi_{1}$ is a left inverse for $\la\mapsto (\va_1, \la \va_2, \la \va_3)$, i.e. $\frac{\la \va_3 - \va_1}{\la \va_2 -1}=\la,$ $\la \in \DD.$ Recall that considering for
any $z\in \EE$ the equation $$\frac{\la z_3 - z_1}{\la z_2 -1}=\la,\quad \la\in \DD,$$ which has a unique solution, one may define a function $\tilde \Phi_1$ being a left inverse of $\va$ and which can be given explicitly
$$\tilde \Phi_1(z):=\frac{2 z_1}{1+z_3+\sqrt{(1 +z_3)^2 - 4z_1 z_2}}.$$

Our aim is to show that there are more left inverses for $\va.$ To do it we shall go to the Cartan symmetric domain of the second type and idea of the proof will rely upon constructing left inverses there. We present below the reasoning which led us to such a construction. So assume that $F:\EE\to \DD$ is holomorphic and such that $F\circ \va=\id.$

Repeating the argument used in Subsection~\ref{ss} involving uniqueness of left inverses in the bidisc we deduce that $$\la\mapsto \left(\frac{(1-\beta^2)\la}{1-\be^2 \la z}, \frac{(1-\beta^2)z}{1-\be^2 \la z}, \frac{\la z-\beta^2}{1-\be^2 \la z} \right)$$ is a complex geodesic in $\EE$ for any $z\in \DD$ and $F$ is its inverse. In particular,
$$ F\left(\frac{(1-\beta^2)\la}{1-\be^2 \la^2\omega^2}, \omega^2 \frac{(1-\beta^2)\la}{1-\be^2 \la^2\omega^2}, \frac{\omega^2 \la^2-\beta^2}{1-\be^2 \la^2 \omega^2}\right)=\la,\quad \la\in \DD, \omega\in \TT.$$ Therefore, $$f_{\omega}:\la\mapsto  \left(2 \frac{(1-\beta^2)\la}{1-\be^2 \la^2\omega^2}, \omega^{-2} \frac{\omega^2 \la^2-\beta^2}{1-\be^2 \la^2 \omega^2}\right)$$ is a complex geodesic in the symmetrized bidisc and $\GG_2\ni (s,p)\mapsto F(\frac s2, \omega^2 \frac s2, \omega^2 p)\in \DD$ is its left inverse, $\omega\in \TT.$ Observe that $f_{\omega} (\la)=\pi(a_{\omega}(\la), b_{\omega}(\la)),$ $\la\in \DD,$ where $a_{\omega}(\la)=\frac{\la -\be\bar\omega}{1-\la \be \omega}$ and $b_{\omega}(\la)=\frac{\la +\be\bar\omega}{1+\la \be \omega}$, $\la\in \DD$.

Making use of the description of holomorphic retracts in the bidisc (see Theorem~\ref{theorem:retracts}) we get the formula (note that $a^{-1}_{\omega} = b_\omega$):
\begin{multline}\label{eq:g-h} F\left(\frac{\lambda_1 +\lambda_2}{2}, \omega^2 \frac{\lambda_1+\lambda_2}{2}, \omega^2 \lambda_1\lambda_2\right)=\\ \frac{ta(\lambda_1) + (1-t) b(\lambda_2)- a(\lambda_1)b(\lambda_2) g(\lambda)}{1-((1-t)a(\lambda_1)+ t b(\lambda_2))g(\lambda)},\end{multline} where $t=t(\omega)\in[0,1]$, $a=a_{\omega}$, $b=b_{\omega}$ and $g=g_\omega$ lies in the closed unit ball of $H^{\infty}(\DD^2)$. We shall construct a left inverse of the above form with $t=1/2.$

Put $g_\omega= \omega^2 h_\omega$, $h=h_\oo$. After some simple calculations we get
\begin{multline}\label{startf} F\left(\frac{\lambda_1 +\lambda_2}{2}, \omega^2 \frac{\lambda_1+\lambda_2}{2}, \omega^2 \lambda_1\lambda_2\right)=\\ \frac{(1-\beta^2)\frac{\la_1 + \la_2}{2} - (\omega^2 \la_1 \la_2 + \beta \omega (\la_1 - \la_2) - \be^2)h}{1-\omega \beta (\la_1 - \la_2) - \be^2 \la_1 \la_2 \omega^2 - (1-\be^2)\frac{\la_1 +\la_2}{2}\omega^2h}.\end{multline}

Observe that for any $h\in \DD$ and $\omega\in \TT$ the mapping in the right side in the formula above lies in the closed unit ball of $H^{\infty}(\DD^2)$, as the mapping in the right side of \eqref{eq:g-h} lies in the closed unit ball of $H^{\infty}(\DD^2)$ for any $g\in \DD$ (see Theorem~\ref{theorem:retracts}). Moreover, the denominator $1-\omega \beta (\la_1 - \la_2) - \be^2 \la_1 \la_2 \omega^2 - (1-\be^2)\frac{\la_1 +\la_2}{2}\omega^2h$ never vanishes if $\lambda_1,\lambda_2\in \DD^2$, $h\in \DD$ and $\omega\in \TT$.  In particular,
\begin{equation}\label{domin}
\left| \frac{(1-\be^2)\frac{\la_1 +\la_2}{2}\omega^2}{1-\omega \beta (\la_1 - \la_2) - \be^2 \la_1 \la_2 \omega^2}\right|\leq 1
\end{equation} providing that $\la_1,\la_2\in \DD$ and $\omega\in \TT$. Note that it is also true for $\omega\in \DD$.
Actually, it suffices to replace $\lambda_i$ with $\omega \lambda_i$, $i=1,2,$

We shall need the following elementary observation:
\begin{lemma}\label{remark}
Let $$\tilde \pi:\CC^3\ni (\la_1, \la_2, \omega)\mapsto \left(\begin{array}{cc}
  (\la_1+\la_2)/2 & \omega(\la_1-\la_2)/2 \\
  \omega(\la_1 - \la_2)/2  & \omega^2 (\la_1+\la_2)/2 \\
\end{array}%
\right)\in \mathcal M_{2\times 2}(\CC).$$ 
A holomorphic function $f$ on $\mathcal R_{II}$ is bounded on $\EE$ if and only if $f\circ \tilde \pi$ is bounded on $\DD^3$. Moreover, the equality 
\begin{equation}\label{sup}\sup_{\mathcal R_{II}} |f|=\sup_{\DD^3} |f\circ \tilde \pi|
\end{equation}
holds.
\end{lemma}
\begin{proof}[Proof of Lemma~\ref{remark}]
First assume that the function $f$ holomorphic on $\mathcal R_{II}$ is additionally continuous on $\overline{\mathcal R_{II}}$.

Note that $\Pi(\tilde \pi(\la_1, \la_2, \omega))=\left((\la_1+\la_2)/2, \oo^2 (\la_1+\la_2)/2, \oo^2 \la_1\la_2\right),$ $\la_1,\la_2,\oo\in \CC$, where $\Pi$ is given by formula~\eqref{eq:coverE}, whence $$\tilde \pi(\DD^3)\subset \mathcal R_{II}.$$ Since $f$ attains its maximum on the Shilov boundary of $\mathcal R_{II}$
and the Shilov boundary of $\mathcal R_{II}$ is contained in $\tilde \pi(\TT^3)$ (for a description of the Shilov boundary of $\EE$ see\cite{You2008} or \cite{Kos2011})
we easily deduce that~\eqref{sup} holds for $f$.

If $f$ is not continuous, to get the assertion apply the previous case to dilatations of the function $f$.
\end{proof}

For $(\la_1,\la_2,\omega)\in \DD^3$ we put $z_1:=(\la_1 +\la_2)/2$, $z_2:=\omega^2(\la_1+\la_2)/2$ and $z_3:=\omega^2 \la_1 \la_2.$ As already mentioned, $z=(z_1,z_2, z_3)\in \EE.$ Moreover, $4(z_1z_2-z_3)=\omega^2 (\la_1-\la_2)^2.$ This and formula for $F$ suggest that $F$ should be considered as the mapping on $\mathcal R_{II}.$

To do so let  us define
\begin{multline}\label{Gh}G_h\left(\begin{array}{cc}
  z_1 & a \\
  a  & z_2 \\
\end{array}%
\right):=\frac{(1-\beta^2)z_1 - (z_3 + 2\be a - \be^2)h }{1-2 a\beta - \be^2 z_3 - (1-\beta^2) z_2 h} =\\ \frac{(1-\beta^2)z_1 - (z_1z_2-(a-\beta)^2)h }{(1-a\beta)^2 - \beta^2 z_1 z_2- (1-\beta^2) z_2 h},\end{multline} where $z_3=z_1z_2 - a^2.$

First observe that if $h$ lies in the closed unit ball of $H^{\infty}(\mathcal R_{II})$ then $G_h$ is well defined on $\mathcal R_{II}$. Actually, it is enough to show that the denominator appearing in formula~\eqref{Gh} never vanishes on $\mathcal R_{II}$. To observe it note that $(1-a\beta)^2 - \beta^2 z_1 z_2 = \det (1- \beta \left(\begin{array}{cc}
  a & z_1 \\
  z_2  & a \\
\end{array}%
\right) )$ so this term does not vanish on $\mathcal R_{II}$, as the spectral norm of the matrix $\left(\begin{array}{cc}
  a & z_1 \\
  z_2  & a \\
\end{array}%
\right)$ is less than $1$ on $\mathcal R_{II}$. Then it suffices to apply Lemma~\ref{remark} to $$ \left(\begin{array}{cc}
  z_1 & a \\
  a  & z_2 \\
\end{array}%
\right)\mapsto \frac{ (1-\beta^2) z_2 }{(1-a\beta)^2 - \beta^2 z_1 z_2}=  \frac{ (1-\beta^2) z_2 }{1- 2 a\beta - \beta^2 z_3}$$ and make use of the inequality~\eqref{domin} to show that the absolute value of the above function is less than $1$ on $\mathcal R_{II}$.

Moreover, applying Lemma~\ref{remark} once again we find that $G_h$ lies in the closed unit ball of $H^{\infty}(\mathcal R_{II})$.

Next note that
\begin{equation}\label{Gh1} G_h\left(\begin{array}{cc}
  (1-\beta^2)\la & \beta \\
  \beta  & 0 \\
\end{array}%
\right)=\la,\quad \la \in \DD,\end{equation} for any $h$ in the closed unit ball of $H^{\infty}(\mathcal R_{II})$.

Finally, some simple computations show that if $h$ is the left inverse for $\la\mapsto \left(\begin{array}{cc}
  (1-\beta^2)\la & -\beta \\
  -\beta  & 0 \\
\end{array}%
\right)$ in $\mathcal R_{II}$, then $G_h$ is its left inverse, as well, i.e.:
\begin{equation}\label{Gh2} G_h\left(\begin{array}{cc}
  (1-\beta^2)\la & -\beta \\
  -\beta  & 0 \\
\end{array}%
\right)=\la,\quad \la \in \DD.\end{equation}

Then for such a left inverse $h$ the following function \begin{multline}\nonumber F_h:\EE\ni (z_1,z_2,z_3)\mapsto  \frac 12 G_h \left(\begin{array}{cc}
  z_1 & \sqrt{z_1 z_2 -z_3^2} \\
  \sqrt{z_1 z_2-z_3^2}  & z_2 \\
\end{array}%
\right)+\\ +\frac 12 G_{h}\left(\begin{array}{cc}
  z_1 & -\sqrt{z_1 z_2 -z_3^2} \\
  -\sqrt{z_1 z_2-z_3^2}  & z_2 \\
\end{array}%
\right)\in \DD\end{multline} is the left inverse for $\la \mapsto ((1-\beta^2)\la, 0, -\beta^2)$ in the tetrablock. Repeating the argument with uniqueness of left inverses in the bidisc we easily deduce that $F_h$ is the left inverse for the geodesic $\va(\la)=
\left(\frac{(1-\be^2)\la}{1-\be^2 \la \tilde Z(\la)}, \frac{(1-\be^2) \tilde Z(\la)}{1-\be^2 \la \tilde Z(\la)}, \frac{\la \tilde Z(\la)-\be^2}{1- \be^2\la \tilde Z(\la)}\right),$ $\la\in \DD$.

So what remains to do is to construct a left inverse $h$ of  $\la\mapsto \left(\begin{array}{cc}
  (1-\beta^2)\la & -\beta \\
  -\beta  & 0 \\
\end{array}%
\right)$ in $\mathcal R_{II}$ such that
\begin{equation}\label{nr}F_h\not\equiv \tilde \Phi_1.
 \end{equation}
We need to introduce some additional notation. Recall that for any $a\in \mathcal R_{II}$ the mapping $$\varphi_a(x)=(1-aa^*)^{-\frac{1}{2}}(x-a)(1-a^*x)^{-1}(1-a^*a)^{\frac{1}{2}},\quad x\in\mathcal R_{II}$$ is an automorphism of $\mathcal R_{II},$ $\varphi_a(0)=-a$ and $\varphi_a(a)=0.$  Let us denote $\varphi_\be:=\varphi_{B(\be)}$, where $B(\be)= \left(\begin{array}{cc}
  0 & \beta \\
  \beta & 0 \\
\end{array}%
\right)$ and $-1<\beta<1.$

To prove \eqref{nr} note that
\begin{equation}\label{Phila}
\tilde \Phi_1(\la, 0, 0)=\la,\quad \la \in \DD.
\end{equation}
Next, observe that $$\tilde h:\mathcal R_{II}\ni \left(\begin{array}{cc}
  z_1 & a \\
  a & z_2 \\
\end{array}%
\right)\mapsto z_1\in \DD$$ is a left inverse of $\DD\ni\la\mapsto \left(\begin{array}{cc}
  \la & 0 \\
  0 & 0 \\
\end{array}%
\right)\in \mathcal R_{II}.$ Therefore, $$h:=\tilde h\circ \varphi_{-\beta}$$ is a left inverse of $\la\mapsto \left(\begin{array}{cc}
  (1-\beta^2)\la & -\be \\
  -\beta & 0 \\
\end{array}%
\right)$ in $\mathcal R_{II}$.
Thus \begin{equation}
h\left(\begin{array}{cc}
  \la & 0 \\
  0 & 0 \\
\end{array}%
\right)=
\tilde h\left(\begin{array}{cc}
  (1-\beta^2)\la & \be \\
  \be & 0 \\
\end{array}%
\right)=
(1-\be^2) \la, \quad \la\in \DD,
\end{equation} whence $F_h(\la,0,0)=G_h\left(\begin{array}{cc}
  \la & 0 \\
  0 & 0 \\
\end{array}%
\right) = \la(1-\be^4)\neq \la= \tilde \Phi_1(\la, 0, 0)$ (by \eqref{Phila}) so actually, $\Phi_h$ gives a new left inverse.

We are left with the case when $\va$ is a complex geodesic of the form~\eqref{formofva},
where
$|c||d|(1+\be^2) = \be$ and $|Z(\la)|<|\la|$, $\la\in \DD\setminus \{0\}$.

Let $F$ be a left inverse of $\va$. We shall show that $F$ is uniquely determined.

As before, one can show that that for any $Z\in \DD$
$$
F\left(
\frac{A(\lambda, Z)(1-\beta^2)}{\Delta(\lambda, Z)},
\frac{C(\lambda, Z)(1-\beta^2)}{\Delta(\lambda, Z)},
\frac{A(\lambda, Z)C(\lambda, Z)-(B(\lambda, Z)+\beta)^2}{\Delta(\lambda, Z)}
\right)=\la,
$$
$\la\in \DD$,
where $A(\la, Z) = a^2 \la + b^2 Z$, $B(\la, Z) = ac \la + bd Z$, $C(\la, Z) = c^2 \la + d^2 Z$ and $\Delta(\la, Z)= (1 + \beta B(\la, Z))^2 - A(\la, Z) C(\la, Z) \beta^2$, $\la, Z\in \DD$.

Note that either $|a|<|b|$ or $|b|<|a|.$ If the first possibility holds we
take $Z=\omega \la,$ where $\omega\in \TT$, and we get that
$$\varphi_\oo: \la\mapsto \left( \frac{2A(\lambda, \la \oo)(1-\beta^2)}{\Delta(\lambda, \la \oo)}, m(\omega)\frac{A(\lambda, \la \oo)C(\lambda, \la \oo)-(B(\lambda, \la \oo)+\beta)^2}{\Delta(\lambda, \la \oo)}
\right),$$
where $m(\omega)=\frac{A(\la, \la \oo)}{C(\la, \la \oo)}$, is a complex geodesic in $\GG_2$ and $$(s,p)\mapsto F\left(\frac{s}{2}, m(\oo) \frac s2,  m(\oo)p\right)$$ is its left inverse, $\oo\in \TT$. In the case $|b|<|a|$ instead of $\varphi_\oo$ one may consider the mapping $$\psi_\oo: \la\mapsto \left( \frac{2C(\lambda, \la \oo)(1-\beta^2)}{\Delta(\lambda, \la \oo)}, n(\omega) \frac{A(\lambda, \la \oo)C(\lambda, \la \oo)-(B(\lambda, \la \oo)+\beta)^2}{\Delta(\lambda, \la \oo)}
\right),$$
where $n(\omega)=\frac{C(\la, \la \oo)}{A(\la, \la \oo)}$. We restrict ourselves to the case $|a|<|b|,$ as for $|b|<|a|$ the proof is the same (with $\psi_\oo$ instead of $\phi_\oo$).

Replacing $\la$ with $\eta_1 \la,$ and $Z$ with $\eta_2 Z,$ where $\eta_1,\eta_2\in \TT$ and using the fact that the tetrablock is $(1,-1,0)$- balanced, we may additionally assume that $a\in \RR_{\geq 0},$ $b=c=\sqrt{1-a^2}$ and $d=-a.$ Then $a^2=\frac{\be^2}{1+\be^2}$, $b^2=\frac{1}{1+\be^2}$ and $m(\oo)= \frac{\oo +\be^2}{1+\be^2 \oo},$ $\oo\in \TT.$ Moreover, $A(\la,\la\omega)=\frac{\la(\beta^2+\omega)}{1+\beta^2}$,
$B(\la,\la\omega)=\frac{\la\beta(1-\omega)}{1+\beta^2}$, $C(\la,\la\omega)=\frac{\la(1+\omega\beta^2)}{1+\beta^2}$, $\triangle(\la,\la\omega)=1+\la\frac{2\beta^2(1-\omega)}{1+\beta^2}-\la^2\beta^2\omega$.
Consequently,
\begin{equation}
 (\va_{\omega}(\la))_1=\frac{2\la \frac{(\beta^2+\omega)(1-\beta^2)}{1+\beta^2}}{1+\la\frac{2\beta^2(1-\omega)}{1+\beta^2}-\la^2\beta^2\omega}, \quad \la\in \DD,
\end{equation}
and
\begin{equation}
(\va_{\omega}(\la))_2=m(\oo) \frac{\la ^2 \oo - \la \frac{2\beta^2(1-\oo)}{1+\beta^2} - \beta^2}{ 1 + \la \frac{2\beta^2 (1-\oo)}{1+ \beta^2} - \la^2 \beta^2 \oo},\quad \la \in \DD.
\end{equation}

Using the description of geodesics in the symmetrized bidisc omitting its royal variety (compare Theorem~\ref{theorem:symmetrizedbidisc-geodesics})
we find that there are M\"obius maps $c_\oo$ and $d_\oo$ such that $\va_\oo =\pi (c_\oo, d_\oo),$ $\oo\in \TT$. In particular, making use of the description of holomorphic retracts in the bidisc we deduce that
\begin{multline}\label{Fs} F\left(\frac{\la_1 +\la_2 }{2}, \frac{1}{m(\oo)} \frac{\la_1+\la_2}{2}, \frac{1}{m(\oo)}\la_1 \la_2 \right)=\\ \frac{t_\oo c^{-1}_\oo(\la_1) + (1-t_\oo) d^{-1}_\oo(\la_2) - c^{-1}_\oo (\la_1) d^{-1}_\oo (\la_2) h_\oo (\la_1, \la_2)}{1-((1-t_\oo) c^{-1}_\oo (\la_1)+ t_{\omega} d^{-1}_\oo (\la_2))h_{\oo}(\la_1, \la_2)},
\end{multline}
 $\la_1,\la_2\in \DD,$ $\oo\in \TT,$ for some $h_\oo\in H^{\infty}(\DD^2)$ such that $||h_\oo||\leq 1$ and $t_\oo\in [0,1]$.

The first crucial observation is that $\pi(c_{-1}, d_{-1})$ is a complex geodesic in $\GG_2$ having one left inverse. More precisely, the equation $c_{-1}=d_{-1}$ has only one solution (which lies in the unit circle) and then we apply Theorem~\ref{theorem:symmetrizedbidisc-leftinverse}.
To prove it let us observe that $$\va_{-1}(\la)=\left(\frac{-2\frac{(1-\be^2)^2}{1+\be^2} \la}{1+\frac{4\be^2}{1+\be^2} \la +\be^2\la^2}, \frac{\la^2+\frac{4\be^2}{1+\be^2} \la +\be^2}{1+\frac{4\be^2}{1+\be^2} \la +\be^2\la^2}\right),\quad \la\in \DD.$$
Let $a:=\frac{2\beta^2}{1+\beta^2}-i\beta\frac{1-\beta^2}{1+\beta^2}\in \DD$ be a solution of the equation $\la^2-\frac{4\be^2}{1+\be^2} \la +\be^2$. Then it is easy to check that $$c_{-1}(\la)=i\frac{|a|}{a}\frac{\la+a}{1+\bar a \la}\quad \text{and}\quad d_{-1}(\la)=-i\frac{a}{|a|}\frac{\la+\bar a}{1+ a \la},\quad \la\in \DD.$$ A direct computation allows us to observe that the equation $c_{-1}(\la)=d_{-1}(\la)$ has one double root $\la=-1$. So actually, the left inverse of $\pi(c_{-1}, d_{-1})$ is uniquely defined. Moreover, it follows from Lemma~\ref{remg2} that
\begin{equation}\label{t-1}t_{-1}=\frac12\end{equation}
and $h_{-1}$ is a unimodular constant.
From Lemma~\ref{remg2} we get that $h_{-1}=-1.$ Putting in \eqref{Fs} $\la_1=\la_2=0$ we find that
\begin{equation}\label{valF}F(0)=-\be^2.
\end{equation}

Differentiating \eqref{Fs} at the point $(c_\oo(0), d_\oo(0))$ with respect to $\la_1$ and $\la_2$ respectively we get that
\begin{multline*}t_\oo (c_\oo^{-1})' (c_\oo(0))= \frac 12 \frac{\partial F}{\partial x_1}(0,0,-\be^2)+\\ \frac{1}{m(\oo)} \frac 12 \frac{\partial F}{\partial x_2}(0,0,-\be^2)+\frac{1}{m(\oo)} d_\oo(0)\frac 12  \frac{\partial F}{\partial x_3}(0,0,-\be^2)
\end{multline*}
and
\begin{multline*}(1-t_\oo) (d_\oo^{-1})' (d_\oo(0))= \frac 12 \frac{\partial F}{\partial x_1}(0,0,-\be^2)+\\ \frac{1}{m(\oo)} \frac 12 \frac{\partial F}{\partial x_2}(0,0,-\be^2)+ \frac{1}{m(\oo)} c_\oo(0)\frac 12  \frac{\partial F}{\partial x_3}(0,0,-\be^2).
\end{multline*}
As $c_\oo(0)+d_\oo(0)=0$
adding the above equalities we find that
\begin{equation}\label{dod} \frac{t_\oo d_\oo'(0) +(1-t_\oo) c_\oo'(0)}{c_\oo'(0) d_\oo'(0) } =  \frac{\partial F}{\partial x_1}(0,0,-\be^2)+\frac{1}{m(\oo)} \frac{\partial F}{\partial x_2}(0,0,-\be^2).
\end{equation}

Differentiating the components of the equality $\pi(c_\oo, d_\oo)=\va_\oo$ we get that
$$c_\oo'(0)+d_\oo'(0)=\frac{2(1-\be^2)}{1+\be^2}(\be^2+\oo)$$
 and (recall $c_\oo (0) + d_\oo (0)=0$ and $c_\oo (0) d_\oo (0)=-m(\oo) \be^2,$ $\oo\in \TT$)
$$c_\oo'(0)- d_\oo'(0) = \pm \frac{2\be}{1+\be^2} (1-\oo) \sqrt{m(\oo)} (1-\be^2),$$ for some branch of the square root $\sqrt{m(\oo)}$.
 The equalities above simply imply that
$$c_\oo'(0) d_\oo'(0) = \left(\frac{1-\be^2}{1+\be^2}\right)\left( (\be^2+\oo)^2 - \be^2 (1-\oo)^2 m(\oo)\right).$$

Let us denote $A:=\frac{\partial F}{\partial x_1}(0,0,-\be^2),$  $B:=\frac{\partial F}{\partial x_2}(0,0,-\be^2)$. Rewriting \eqref{dod} we get that
$$\frac{(t_\oo-\frac12)(d_\oo'(0) - c_\oo'(0))+\frac12(c_\oo'(0)+d_\oo'(0))}{c_\oo'(0)d_\oo'(0)} = A+\frac{1}{m(\oo)} B.
$$
Some simple computations lead to
\begin{multline}\label{form}
 \pm 2(t_\oo -\frac12) \sqrt{m(\oo)}  \frac{1-\be^2}{1+\be^2} \be (1-\oo) = -\frac{1-\be^2}{1+\be^2}(\be^2 +\oo)+\\ \left(\frac{1-\be^2}{1+\be^2} \right)^2 \left( (\be^2+\oo)^2 - \be^2 m(\oo) (1-\oo)^2 \right) \left( A+\frac{1}{m(\oo)} B\right).
\end{multline}
Putting here $\oo=-1$ and making use of \eqref{t-1} we find that $$A-B=\frac{-1}{1+\be^2}.$$ On the other hand putting in \eqref{form} $\oo=1$ we get $$A+B=\frac{1}{1-\be^2}.$$

Looking at \eqref{form} carefully we see that there are holomorphic mappings $\va$ and $\psi$ in a neighborhood of $\overline\DD$, $\psi(-\beta^2)\neq 0$ (which may be given explicitly) such that $$\pm (t_\oo-\frac12) \sqrt{m(\oo)}= \va(\oo) m(\oo) + \psi(\oo),\quad \oo\in \TT.$$
Raising the equation above to the second power we get
\begin{equation}\label{pow}(t_\oo-\frac12)^2 m(\oo)= (\va(\oo) m(\oo) + \psi(\oo))^2,\quad \oo\in \TT.
\end{equation}
Using the fact that $t_\oo$ attains only real values we easily find that there are $\alpha'\in \CC$ and $\alpha''\in \RR$ such that $(t_\oo-\frac{1}{2})^2=\overline{\alpha'}m(\oo)^{-1}+ \alpha'' + \alpha' m(\oo)$ (this observation appears in the literature as the Gentili's lemma, see \cite{Gen 1987}). Since $t_{-1}=1/2$ we get that $\alpha''=2\Re \alpha'.$ Comparing multiplicities of roots we deduce that $\alpha'\in \RR$ and $\alpha''=2\alpha'.$

Concluding, we have shown the existence of $\alpha$ such that $\pm(t_\oo - \frac12) \sqrt{m(\oo)} = \alpha (m(\oo)+1),$ $\oo\in \TT.$
Note that $\alpha=\psi(-\be^2)$ and the last value is equal to $\frac{-\be}{2(1+\be^2)}.$

Let $\ga_\oo$ and $\de_\oo$ be that $c_\oo (\ga_\oo)=d_\oo (\de_\oo)=0$. Our aim is to make use of the formula \eqref{Fs} with $\la_1=\la_2=0.$ Note that it may be rewritten as (remember about \eqref{valF}):
\begin{equation}\label{Fs1}
 \frac{\frac12\left(\de_\oo^{-1} + \gamma_\oo^{-1}\right) + (t_\oo-\frac12) \left(\de_\oo^{-1} - \gamma_\oo^{-1} \right) - h_\oo(0)}{\ga_\oo^{-1}\de_\oo^{-1} - \left(\frac12\left(\de_\oo^{-1} + \gamma_\oo^{-1}\right) - (t_\oo-\frac12) \left(\de_\oo^{-1} - \gamma_\oo^{-1}\right)     \right) h_\oo(0)}= -\be^2.
\end{equation}

Recall that
$\Delta(\la, \oo\la)=1 +2 \frac{\be^2}{1+\be^2} (1-\oo) \la - \la^2 \be^2 \oo.$ Certainly, $\frac{1}{\bar \ga_\oo}$ and $\frac{1}{\bar \de_\oo}$ solve the equation $$\Delta (\la, \oo \la)=0.$$
 From this we easily find that $$\left(t_{\omega}-\frac12\right)^2 (\ga_\oo^{-1} - \de_\oo^{-1})^2= \left(\frac{1+\oo}{1+\be^2}\right)^2.$$
Therefore $\left(t_{\omega}-\frac12\right) (\ga_\oo^{-1} - \de_\oo^{-1})=s_\oo,$ where either $s_\oo= \frac{1+\oo}{1+\be^2}$ or $s_\oo= -\frac{1+\oo}{1+\be^2}$. Putting it to \eqref{Fs1} we find that
$$h_\oo(0) \left(1+\be^2\frac{\oo -1}{1+\be^2} - \be^2 s_\oo\right) = \frac{\oo-1}{1+\be^2}-\oo + s_\oo.$$ If $s_\oo= -\frac{1+\oo}{1+\be^2}$ we get that $|h_\oo(0)|>1,$ $\oo\in \TT\setminus\{-1\}$, which is impossible. Therefore $s_\oo=\frac{1+\oo}{1+\be^2}$ and $h_\oo(0)=\oo,$ $\oo\in \TT$. The maximum principle implies that that $h_\oo$ is constant, $\oo\in \TT$, and that $t_\oo$ is also uniquely determined.

Thus $F$ is uniquely defined on $\{((\la_1 + \la_2)/2, (\la_1 + \la_2)/(2m(\oo)), \la_1 \la_2 /m(\oo)):\ \la_1, \la_2\in \DD, \oo\in \TT\}$. Thus the assertion follows from the identity principle.

\bibliographystyle{amsplain}

\end{document}